\documentclass[12pt,oneside]{amsart}
\usepackage{bm, calc,geometry, verbatim, graphicx, amssymb, color, mathabx, mathtools, enumitem, bm, mathrsfs, amsmath, amsthm, ulem, xspace,tabularx }
\usepackage[usenames,dvipsnames,svgnames,table]{xcolor}
\usepackage[pdftex,bookmarks,colorlinks,breaklinks]{hyperref}  
\usepackage{rotating}
\usepackage{adjustbox}
\usepackage{blindtext}
\usepackage{relsize}

\hypersetup{linkcolor=blue,citecolor=red,filecolor=dullmagenta,urlcolor=darkblue} 
\newtheorem{theorem}{Theorem}[section]
\newtheorem{corollary}[theorem]{Corollary}
\newtheorem{lemma}[theorem]{Lemma}

\newtheorem{remark}[theorem]{Remark}

\newtheorem{definition}[theorem]{Definition} 

\newcommand\mult{\operatorname{\textup{{\fontfamily{ptm}\selectfont mult}}}}
\newcommand\dg{\operatorname{\textup{{\fontfamily{ptm}\selectfont deg}}}}

\newcommand\brho{\operatorname{\boldsymbol{\rho}}}

      \makeatletter
      \def\@setcopyright{}
      \def\serieslogo@{}
      \makeatother      

\begin{document}
   \author{Amin  Bahmanian}
   \address{Department of Mathematics,
  Illinois State University, Normal, IL USA 61790-4520}

\title[Ryser's Theorem for  $\brho$-latin Rectangles]{Ryser's Theorem for  $\brho$-latin Rectangles}

   \begin{abstract}  
Let $L$ be an $n\times n$ array whose  top left $r\times s$ subarray is filled with $k$ different symbols, each occurring at most once in each row and at most once in each column. We find necessary and sufficient conditions that ensure  the remaining cells of $L$ can be filled  such  that each symbol occurs at most once in each row and at most once in each column, and 
each symbol occurs a prescribed number of times in $L$. The case where the prescribed number of times each symbol occurs is $n$ was solved by Ryser  (Proc. Amer. Math. Soc. 2 (1951), 550--552), and the case $s=n$ was settled by Goldwasser et al. (J. Combin. Theory Ser. A 130 (2015), 26--41). Our technique leads to a very  short proof of the latter.
   
   \end{abstract}

   %Primary 00A30; Secondary 00A22, 03E20}

   \date{\today}

   \maketitle   
\section{Introduction} 
Let $n,k\in \mathbb{N}$ and let $\brho=(\rho_1,\dots,\rho_k)$ with $1\leq \rho_\ell\leq n\leq k$ for $\ell\in [k]:=\{1,2,\dots, k\}$ such that $\sum_{\ell\in [k]} \rho_\ell=n^2$. A {\it $\brho$-latin square} $L$ of order $n$ is an $n\times n$ array filled with $k$ different {\it symbols}, $[k]$, each occurring at most once in each row and at most once in each column, such that each symbol $\ell$ occurs exactly $\rho_\ell$ times in $L$ for $\ell\in [k]$. An $r \times s$ {\it  $\brho$-latin rectangle} on the set $[k]$ of symbols is an $r\times s$ array in which each symbol in $[k]$ occurs at most once in each row and in each column, and in which   each symbol $\ell$ occurs at most $\rho_\ell$ times  for $\ell\in [k]$. A {\it latin square} (or {\it rectangle}) is a $\brho$-latin square (or {\it rectangle}) with $\brho=(n,\dots,n)$. 

We are interested in conditions that ensure that an $r \times s$ $\brho$-latin rectangle can be extended to a $\brho$-latin square.  The cases   $(\brho,s)=((n,\dots,n),n)$, and $\brho=(n,\dots,n)$ were settled by Hall \cite{MR13111}, and  Ryser \cite{MR42361}, respectively. Most recently, the case $s=n$ was solved by Goldwasser,  Hilton,  Hoffman, and \"{O}zkan  \cite{MR3280683} (See Theorem \ref{rhohallthmgoldetal}).  For an old yet excellent survey on embedding latin squares, we refer the reader to  Lindner \cite{MR1096296}. 

We need a few pieces of notation before we can state our main result. The complement of a set $S$ is denoted by $\bar S$,  $x \dotdiv y:=\max\{0, x-y\}$, and we write $x-y\dotdiv z$ for $(x-y)\dotdiv z$. If $L$ is an $r\times s$  $\brho$-latin rectangle, $e_\ell$ will denote the number of occurrences of symbol $\ell$ in $L$. We will always assume that $n>\min \{r,s\}$, and that $[r], [s]$, and $[k]$ are the set of rows, the set of columns, and the set of symbols of $L$. 
Let $i$ be a row in $[r]$,  $j$ be a column in $[s]$,  $\ell$ be a symbol in $[k]$, and let $I\subseteq [r], J\subseteq [s], K\subseteq [k]$. Then $\mu_K(i), \mu_K(j), \mu_I(\ell)$, and $\mu_J(\ell)$ are the number of symbols in $K$ that are missing in row $i$, the number of symbols in $K$ that are missing in column $j$, the number of rows in $I$ where symbol $\ell$ is missing, and the number of columns in $J$ where symbol $\ell$ is missing, respectively, and 
\begin{align*}
    \mu_K(I)&:=\sum_{i\in I}\mu_K(i), & \mu_I(K)&:=\sum_{\ell\in K} \mu_I(\ell),\\
    \mu_K(J)&:=\sum_{j\in J}\mu_K(j), & \mu_J(K)&:=\sum_{\ell\in K} \mu_J(\ell).
\end{align*}
Observe that 
\begin{align*}
    &\mu_K(I)= \mu_I(K), &\mu_{I}(\ell)+\mu_{\bar I}(\ell)=r-e_\ell,\\
    &\mu_K(J)= \mu_J(K), &\mu_{J}(\ell)+\mu_{\bar J}(\ell)=s-e_\ell.
\end{align*}

Using our notation, the main  result of  \cite{MR3280683} can be stated  as follows. 
\begin{theorem} \cite{MR3280683} \label{rhohallthmgoldetal}
An $r\times n$  $\brho$-latin rectangle $L$ can be completed to an $n\times n$ $\brho$-latin square  if and only if $\rho_\ell - e_\ell\leq n-r$ for $\ell\in [k]$, and either one of the following conditions hold.
\begin{align}
    |J|(n-r) &\leq \sum_{\ell\in [k]} \min \Big\{\rho_\ell-e_\ell, \mu_J(\ell)\Big\}  &\forall J\subseteq  [n],\label{orenec}\\
    |J|(n-r) &\geq \sum_{\ell\in [k]} \Big( \rho_\ell-e_\ell\dotdiv \mu_{\bar J}(\ell)\Big)  &\forall J\subseteq [n]. \label{hoffmanc}
\end{align}
\end{theorem}
Rewriting \eqref{orenec} for $\bar J$ will result in \eqref{hoffmanc}, so conditions \eqref{orenec} and \eqref{hoffmanc} are essentially the same (See Remark \ref{ordenewcond}). In \cite{MR3280683}, \eqref{orenec} is being referred to as Hall's refined condition, which  is a special case of the more general Hall's constrained condition. If you think of a set $J$ of columns of $L$ as an $r\times |J|$ subrectangle of $L$, Hall's constrained condition checks an  inequality similar to that of Hall's refined condition for all subsets of cells, rather than those subsets that form an $r\times |J|$ rectangle.

We provide a short proof of   Theorem \ref{rhohallthmgoldetal} (See Theorem \ref{rhohallthm}). We also generalize Theorem \ref{rhohallthmgoldetal} further by establishing an analogue of Ryser's Theorem for $\brho$-latin rectangles. 

Suppose that an $r\times s$  $\brho$-latin rectangle $L$ is extended to an $n\times n$ $\brho$-latin square, and let us fix a symbol $\ell\in [k]$. On the one hand, there are $\rho_\ell-e_\ell$ occurrences of symbol $\ell$ outside the top left $r\times s$ subsquare. On the other hand, there are at most $n-r$ occurrences of symbol $\ell$ in the last $n-r$ rows and at most $n-s$ occurrences of symbol $\ell$ in the last $n-s$ columns. Therefore, we must have
\begin{align} \label{mainryserineq}
e_\ell \geq r+s+\rho_\ell - 2n\quad \forall \ell\in [k].
\end{align}
It turns out that condition \eqref{mainryserineq} is a special case of a more general necessary condition. Let 
\begin{align*}
    P_t=\{\ell \in [k] \ |\ \rho_\ell-e_\ell >  n-t  \}\quad \mbox{for } t\in \{r,s\}.
\end{align*}
Note that $P_{\min \{r,s\}}\subseteq P_{\max \{r,s\}}$.

\begin{definition}  \label{fitdef} A sequence $\{(a_\ell,b_\ell)\}_{\ell=1}^k$ is said to be {\it fitting} if $a_\ell,b_\ell \in\mathbb{N}\cup \{0\}$ for  $\ell\in [k]$, and the following conditions are satisfied.
\begin{align*}
\begin{cases}
       \displaystyle\sum_{\ell\in [k]} a_\ell=r(n-s)+|P_r|(n-r)-\sum_{\ell\in P_r}(\rho_\ell-e_\ell), \\[20pt]
        \displaystyle\sum_{\ell\in [k]} b_\ell=s(n-r)+|P_s|(n-s)-\sum_{\ell\in P_s}(\rho_\ell-e_\ell),\\[20pt]
         a_\ell+b_\ell\leq 
        \begin{cases}
            2n- r-s+ e_\ell-\rho_\ell& \mbox {if  } \ell \in P_{\min \{r,s\}},\\ 
            n- \max \{r,s\} & \mbox {if  } \ell \in P_{\max \{r,s\}}\backslash P_{\min \{r,s\}},\\
            \rho_\ell -e_\ell & \mbox {if  } \ell \in \bar P_{\max \{r,s\}}.
        \end{cases}
    \end{cases}
\end{align*} 
\end{definition}
Here is our main result. 
\begin{theorem} \label{rhoryserthmfullversion}
An $r\times s$  $\brho$-latin rectangle $L$ can be completed to an $n\times n$ $\brho$-latin square  if and only if
 there exists a fitting sequence $\{(a_\ell,b_\ell)\}_{\ell=1}^k$ 
such that any of the  following conditions
\begin{align*}
    |I|(n-s) &\leq \sum_{\ell\in P_r} \min \Big\{a_\ell+\rho_\ell-e_\ell-n+r, \mu_I(\ell)\Big\}+ \sum_{\ell\in \bar P_r} \min \Big\{a_\ell, \mu_I(\ell)\Big\}  &\forall I\subseteq  [r],\\
    \sum_{\ell\in (K\cap P_r)} (\rho_\ell-e_\ell) &\leq \sum_{i\in [r]} \min \Big\{n-s, \mu_K(i)\Big\}- \sum_{\ell\in K} a_\ell +|K\cap P_r|(n-r) &\forall K\subseteq  [k],\\
  |I|(n-s) &\geq \sum_{\ell\in P_r} \Big( a_\ell+\rho_\ell+\mu_I(\ell)\dotdiv n\Big)+\sum_{\ell\in \bar P_r} \Big( a_\ell\dotdiv \mu_{\bar I}(\ell)\Big)  &\forall I\subseteq [r],\\    
     \sum_{\ell\in (K\cap P_r)}(\rho_\ell-e_\ell)    &\geq \sum_{i\in [r]} \Big (n-s\dotdiv \mu_{\bar K}(i)\Big) + |K\cap P_r|(n-r) - \sum_{\ell\in K} a_\ell &\forall K\subseteq [k],\\
    |I|(n-s) &\leq \sum_{\ell\in (K\cap P_r)} (\rho_\ell-e_\ell) +\sum_{\ell\in K} a_\ell+ \mu_I(\bar K) -|K\cap P_r|(n-r)  &\forall I\subseteq [r], K\subseteq [k], \\
    \sum_{\ell\in (K\cap P_r)} (\rho_\ell-e_\ell) &\leq |I|(n-s)+|K\cap P_r|(n-r) + \mu_K(\bar I)- \sum_{\ell\in K} a_\ell   &\forall I\subseteq  [r], K\subseteq [k],
\end{align*}
together with any of the following conditions hold.
\begin{align*}
    |J|(n-r) &\leq \sum_{\ell\in P_s} \min \Big\{b_\ell+\rho_\ell-e_\ell-n+s, \mu_J(\ell)\Big\} + \sum_{\ell\in \bar P_s} \min \Big\{b_\ell, \mu_J(\ell)\Big\}  &\forall J\subseteq  [s],\\
    \sum_{\ell\in (K\cap P_s)} (\rho_\ell-e_\ell) &\leq \sum_{j\in [s]} \min \Big\{n-r, \mu_K(j)\Big\}- \sum_{\ell\in K} b_\ell +|K\cap P_s|(n-s)  &\forall K\subseteq  [k],\\
  |J|(n-r) &\geq \sum_{\ell\in [k]} \Big( b_\ell+\rho_\ell+\mu_J(\ell)\dotdiv n\Big) +\sum_{\ell\in \bar P_s} \Big( b_\ell\dotdiv \mu_{\bar J}(\ell)\Big)  &\forall J\subseteq [s],\\    
     \sum_{\ell\in (K\cap P_s)}(\rho_\ell-e_\ell)    &\geq \sum_{j\in [s]} \Big (n-r\dotdiv \mu_{\bar K}(j)\Big) + |K\cap P_s|(n-s) - \sum_{\ell\in K} b_\ell &\forall K\subseteq [k],\\
    |J|(n-r) &\leq \sum_{\ell\in (K\cap P_s)} (\rho_\ell-e_\ell) + \sum_{\ell\in K} b_\ell+ \mu_J(\bar K)-|K\cap P_s|(n-s)  &\forall J\subseteq [s], K\subseteq [k], \\
    \sum_{\ell\in (K\cap P_s)} (\rho_\ell-e_\ell) &\leq |J|(n-r)+|K\cap P_s|(n-s) + \mu_K(\bar J) - \sum_{\ell\in K} b_\ell &\forall J\subseteq  [s], K\subseteq [k].
\end{align*}
\end{theorem}

For a much simpler generalization of Ryser's Theorem, see Corollary \ref{cor2}.
\section{Tools}
For $x,y\in \mathbb{R}$, $x\approx y$ means $\lfloor y \rfloor \leq x\leq \lceil y \rceil$. For a real-valued function $f$ on a domain $D$ and  $S\subseteq D$,  $f(S) :=\sum_{x\in S}f(x)$.

All graphs under consideration are loopless, but they may have parallel edges.  For a graph $G=(V,E)$, $u,v\in V$ and $S,T\subseteq V$ with $S\cap T$, $\dg_G(u)$, $\mult_G(uv)$,  $\mult_G(uS)$, and $\mult_G(ST)$ denote the number of edges incident with $u$,  the number of edges whose endpoints are $u$ and $v$,  the number of edges between $u$ and $S$, and the number of edges between $S$ and $T$, respectively. If the edges of $G$ are colored with $k$ colors (the set of colors is always $[k]$), then $G(\ell)$ is the color class $\ell$ of $G$ for $\ell\in [k]$. A bigraph $G$ with bipartition $\{X, Y \}$ will be denoted by $G[X, Y ]$, and for $S\subseteq X$, $\bar S:=X\backslash S$.

Let $G$ be a graph whose edges are colored, and let $\alpha\in V(G)$. By {\it splitting} $\alpha$ into $\alpha_1,\dots,\alpha_p$, we obtain a new graph $F$ whose vertex set is $\left(V(G)\backslash \{\alpha\}\right) \cup \{\alpha_1,\dots, \alpha_p\}$ so that each edge $\alpha u$ in $G$ becomes $\alpha_i u$ for some $i\in [p]$ in $F$. Intuitively speaking, when we {\it split}  a vertex $\alpha$ into $\alpha_1,\dots,\alpha_p$, we share the edges incident with $\alpha$ among $\alpha_1\dots,\alpha_p$.   In this manner, $F$ is a {\it detachment} of $G$, and $G$ is an {\it amalgamation} of $F$ obtained by {\it identifying} $\alpha_1,\dots, \alpha_p$ by $\alpha$. The following detachment lemma will be crucial in the proof of our main result. 
\begin{lemma} 
\label{amalgambahrod} \cite{MR2946077}
  Let $G$ be a graph whose edges are colored with $k$  colors, and let $\alpha,\beta$ be two vertices of $G$. There exists a graph $F$ obtained by splitting $\alpha$ and $\beta$ into $\alpha_1,\dots,\alpha_p$, and  $\beta_1,\dots,\beta_q$, respectively, such that the following conditions hold.
\begin{enumerate}
     \item [\textup{(i)}]  $\dg_{F(\ell)}(\alpha_i)\approx \dg_{G(\ell)}(\alpha)/p$  for  $i\in [p],\ell\in [k]$;
     \item [\textup{(ii)}]  $\dg_{F(\ell)}(\beta_j)\approx \dg_{G(\ell)}(\beta)/q$ for  $j\in [q],\ell\in [k]$;
     \item [\textup{(iii)}]  $\mult_F(\alpha_i u)\approx \mult_G(\alpha u)/p$ for  $i\in[p],u\in V(G)\backslash \{\alpha,\beta\}$;
     \item [\textup{(iv)}]  $\mult_F(\beta_j u)\approx \mult_G(\beta u)/q$ for  $j\in[q],u\in V(G)\backslash \{\alpha,\beta\}$;
     \item [\textup{(v)}]  $\mult_F(\alpha_i \beta_j)\approx \mult_G(\alpha\beta)/(pq)$ for  $i\in[p],j\in[q]$.
\end{enumerate}

\end{lemma}
For a hypergraph analogue, we refer the reader to \cite{MR2942724}. To give the reader an idea about the usefulness of this detachment lemma, let us show how to construct $\brho$-latin squares.   Theorem \ref{rholatconstthm} can be viewed as an immediate consequence of Theorem \ref{rhohallthmgoldetal} (See  \cite[Theorem 5.1]{MR3280683}). 
\begin{theorem} \label{rholatconstthm} For every  $n,k, \brho:=(\rho_1,\dots,\rho_k)$ with $1\leq \rho_1,\dots,\rho_k\leq n\leq k$ and $\sum_{\ell\in [k]} \rho_\ell=n^2$,  there exists a  $\brho$-latin square of order $n$.
\end{theorem}
\begin{proof}
Let  $G[\{\alpha\}, \{\beta\}]$ be a $k$-edge-colored bigraph with $\mult_G(\alpha \beta)=n^2$, $\mult_{G(\ell)}(\alpha \beta)=\rho_\ell$ for $\ell\in [k]$. 
Applying the detachment lemma with $p=q=n$ yields the complete bigraph $F\cong K_{n,n}$ whose colored edges corresponds to symbols in the desired $\brho$-latin square of order $n$. 
\end{proof}
 Let $f,g$ be  integer functions on the vertex set of a graph $G$ such that $0\leq g(x)\leq f(x)$ for all $x$. A {\it $(g,f)$-factor} is a spanning subgraph $F$ of $G$ with the property that $g(x)\leq \dg_F(x)\leq f(x)$ for each $x$, and an {\it $f$-factor} is  an $(f,f)$-factor.  We need the following result which is known as  Ore's Theorem. For far reaching generalizations of Ore's Theorem, we refer the reader to  Lov\'{a}sz's seminal paper \cite{MR325464}. 
\begin{theorem}\cite{MR83725} \label{OresThm}
The bipgraph $G[X,Y]$ has an $f$-factor  if and only if $f(X)=f(Y)$ and either one of the following conditions hold.
\begin{align*}
    f(A)&\leq \sum_{u\in Y} \min\Big\{f(u), \mult_G(uA)\Big\} &\forall A\subseteq X,\\
    f(A) &\leq f(B) + \mult_G(A\bar B) &\forall A\subseteq X, B\subseteq Y.
\end{align*}
\end{theorem} 
\begin{remark} \label{ordenewcond} \textup{
Let us fix $A\subseteq X$. For $u\in Y$, we have
\begin{align*}
     \kappa (u)&:=\min\Big\{f(u), \mult_G(uA)\Big\}+ \Big (f(u)\dotdiv \mult_G(u  A)\Big)\\
     &=
     \begin{cases}
        \mult_G(uA)+ \Big (f(u)- \mult_G(u  A)\Big) & \mbox{if  } f(u)\geq  \mult_G(uA) \\
        f(u)+0 & \mbox{if  } f(u)<  \mult_G(uA)
     \end{cases}\\
     &=f(u).
\end{align*}
So, $\sum_{u\in Y} \kappa(u)=f(Y)$, and the first condition in Theorem \ref{OresThm} is equivalent to
\begin{align*}
        f(A)&\leq f(Y)-\sum_{u\in Y} \Big (f(u)\dotdiv \mult_G(u  A)\Big),
\end{align*}
which is equivalent to 
\begin{align*}
             f(\bar A)\geq     \sum_{u\in Y} \Big (f(u)\dotdiv \mult_G(u  A)\Big).
\end{align*}
}\end{remark}

\section{{\normalfont Hall's Theorem for }\protect\boldmath$\rho${\normalfont-latin Rectangles} }
Hall \cite{MR13111} showed that any $r\times n$ latin rectangle can be extended to an $n\times n$ latin square. Here we give a   short proof of the main result of \cite{MR3280683} which generalizes Hall's theorem to $\brho$-latin rectangles.

\begin{theorem} \label{rhohallthm}
An $r\times n$  $\brho$-latin rectangle $L$ can be completed to an $n\times n$ $\brho$-latin square  if and only if $\rho_\ell - e_\ell\leq n-r$ for  $\ell\in [k]$, and any of the following conditions hold.
\begin{align*}
    |J|(n-r) &\leq \sum_{\ell\in [k]} \min \Big\{\rho_\ell-e_\ell, \mu_J(\ell)\Big\}  &\forall J\subseteq  [n],\\
    \sum_{\ell\in K} (\rho_\ell-e_\ell) &\leq \sum_{j\in [n]} \min \Big\{n-r, \mu_K(j)\Big\}  &\forall K\subseteq  [k],\\%\label{orenec1equiv}\\
  |J|(n-r) &\geq \sum_{\ell\in [k]} \Big( \rho_\ell-e_\ell\dotdiv \mu_{\bar J}(\ell)\Big)  &\forall J\subseteq [n],\\    
     \sum_{\ell\in K}(\rho_\ell-e_\ell)    &\geq \sum_{j\in [n]} \Big (n-r\dotdiv \mu_{\bar K}(j)\Big)  &\forall K\subseteq [k],\\%\label{anothercondequi}\\
%    |J|(n-r) &\geq \sum_{\ell\in [k]} \Big( \rho_\ell+\mu_J(\ell)\dotdiv n\Big) \quad &\forall J\subseteq [n], \label{hoffmancond}\\
    |J|(n-r) &\leq \sum_{\ell\in K} (\rho_\ell-e_\ell) + \mu_J(\bar K)  &\forall J\subseteq [n], K\subseteq [k], \\%\label{orecond2}\\
    \sum_{\ell\in K} (\rho_\ell-e_\ell) &\leq |J|(n-r) + \mu_K(\bar J)  &\forall J\subseteq  [n], K\subseteq [k].%\label{orecond2more},\\
\end{align*}
\end{theorem}
 \begin{proof} Recall that the set of symbols used is $[k]$. The  necessity of 
\begin{align} \label{ryscond1}
    \rho_\ell - e_\ell&\leq n-r\quad \forall \ell\in [k]
\end{align}
is immediate from ~\eqref{mainryserineq}.  The necessity of the remaining conditions will be evident at the end of the proof. For now, let us assume that \eqref{ryscond1} is satisfied.  Let $F[\tilde X, Y]$ be the complete bigraph $K_{n,n}$ where $\tilde X:=\{x_1,\dots,x_n\}, Y:=\{y_1,\dots, y_n\}$, and let $X=\{x_1,\dots, x_r\}$. The edge $x_i y_{j}$ of $F$ is colored $\ell$ if $L_{ij}=\ell$ for $i\in[r],j\in[n]$. We have $\dg_{F(\ell)}(u)\leq 1$ for $u\in X\cup Y,$    $e_\ell=|E(F(\ell))|\leq \rho_\ell$ for $\ell\in [k]$, and $\sum_{\ell\in [k]}e_\ell=rn$.  
Let  $G$ be the bigraph obtained from $F$ by amalgamating $x_{r+1},\dots, x_n$ into a single vertex $\alpha$, so $\mult_G(\alpha y_j)=n-r$ for $j\in [n]$. Let  $\Gamma [Y, [k]]$ be the simple bigraph whose edge set is $$\{u \ell\ \big|\ u\in Y, \ell\in [k],  \dg_{F(\ell)}(u)=0\}.$$ 
For    $j\in [n]$,  $\sum_{\ell\in [k]}\dg_{F(\ell)}(y_j)=r$, and for $\ell\in [k]$, there are $e_\ell$ edges in $F(\ell)$. Therefore, 
\begin{align*}%\label{nothingserious}
\begin{cases}
\dg_\Gamma(y_j)=k-r & \mbox{if  } j\in [n], \\
\dg_{\Gamma}(\ell)=n-e_\ell & \mbox{if  } \ell\in [k].
\end{cases}
\end{align*}
Observe that $L$ can be completed if and only if the uncolored edges of $F$ can be colored so that
\begin{align}\label{colorcon1s}
\forall \ell\in [k], \quad \quad \begin{cases}
\dg_{F(\ell)}(u)\leq 1 & \mbox{if  } u\in \tilde X\cup Y, \\
|E(F(\ell))|=\rho_\ell.
\end{cases}
\end{align}
We show that the coloring of  $F$ can be completed such that \eqref{colorcon1s} holds if and only if the coloring of $G$ can be completed so that 
\begin{align}\label{colorcon2s}
\forall \ell\in [k],\quad \quad  \begin{cases}
\dg_{G(\ell)}(u) \leq 1 & \mbox{if  } u \in X \cup Y, \\
\dg_{G(\ell)}(\alpha) \leq n-r, \\
|E(G(\ell))|= \rho_\ell.
\end{cases}
\end{align}
To see this, first assume that the coloring of  $F$ can be completed such that \eqref{colorcon1s} holds. Identifying all the vertices in $\tilde X\backslash X$ of $F$ by $\alpha$ we will get the graph $G$ satisfying \eqref{colorcon2s}. Conversely, suppose that  we have a  coloring of $G$
such that \eqref{colorcon2s} holds. Applying   Lemma \ref{amalgambahrod} to $G$, we get a graph $F'$ obtained by splitting $\alpha$  into $\alpha_1,\dots,\alpha_{n-r}$ such that the following hold.
\begin{enumerate}
    \item  [(i)] $\dg_{F'(\ell)}(\alpha_i)\approx\dg_{G(\ell)}(\alpha)/(n-r)\leq 1$  for  $i\in [n-r],\ell\in [k]$;
    \item  [(ii)] $\mult_{F'}(\alpha_i u)= \mult_G(\alpha u)/(n-r)=1$ for  $i\in[n-r],u\in Y$.
\end{enumerate}
Since $F'\cong F$ and the coloring of $F'$ satisfies \eqref{colorcon1s}, we are done. 

Now, we show that the coloring of  $G$ can be completed such that  \eqref{colorcon2s} is satisfied if and only if there exists  a subgraph $\Theta$ of $\Gamma$  with $n(k-r)$ edges so that 
\begin{align}\label{colorcon3s}
\begin{cases}
\dg_\Theta (y_j)=n-r  & \mbox{if  }  j\in [n], \\
\dg_\Theta (\ell)=\rho_\ell-e_\ell  & \mbox{if  } \ell \in [k].
\end{cases}
\end{align}
To prove this, suppose that the coloring of  $G$ can be completed such that \eqref{colorcon2s} holds. Let  $\Theta [ Y, [k]]$ be the bigraph whose edge set is 
$$\{u\ell\ \big|\ u\in  Y, \ell\in [k], \mbox{ and  }\alpha u\in E(G(\ell))   \}.$$ 
Observe that $\Theta \subseteq \Gamma$. Moreover, $\dg_\Theta(y_j)=\mult_G(\alpha y_j)=n-r$ for $j\in [n]$, and $\dg_{\Theta}(\ell)= |E(G(\ell))|-e_\ell=\rho_\ell-e_\ell$ for $\ell\in [k]$, and so \eqref{colorcon3s} holds.  Conversely, suppose that   $\Theta\subseteq \Gamma$ satisfying \eqref{colorcon3s} exists. For each $\ell\in [k]$,   if $\ell y_j \in E(\Theta)$ for some $j\in [n]$, we color an $\alpha y_j$-edge in $G$ with $\ell$. Since $\dg_\Theta(y_j)=n-r$ for $j\in [n]$, all the edges between $\alpha$ and $Y$ can be colored this way. Since $\Theta$ is simple, $d_{G(\ell)}(u)\leq 1$ for $\ell\in [k]$ and $u\in  Y$.  It is also clear that $|E(G(\ell))|= \rho_\ell-e_\ell+e_\ell=\rho_\ell$, and by ~\eqref{ryscond1} 
$\dg_{G(\ell)}(\alpha)=\dg_\Theta(\ell)=\rho_\ell-e_\ell \leq n-r$ for $\ell\in [k]$.

Let 
\begin{align*}
    \begin{cases}
        f: V(\Gamma)\rightarrow \mathbb{N}\cup \{0\},\\
        f(y_j)=n-r & \mbox{for  } j\in [n],\\
        f(\ell)=\rho_\ell-e_\ell & \mbox{for  }\ell\in [k].      
    \end{cases}
\end{align*}
 Since $f(Y)= f([k])$, by Ore's Theorem and Remark \ref{ordenewcond}, $\Gamma$ has an $f$-factor (and so $\Theta\subseteq \Gamma$ satisfying \eqref{colorcon3s} exists)  if and only if any of the following  conditions holds. 
\begin{align*}
    f(J)&\leq \sum_{\ell\in [k]} \min\Big\{f(\ell), \mult_\Gamma(\ell J)\Big\} &\forall J\subseteq Y,\\
    f(K)&\leq \sum_{u\in Y} \min\Big\{f(u), \mult_\Gamma(uK)\Big\} &\forall K\subseteq [k],\\
    f(J)&\geq \sum_{\ell\in [k]} \Big (f(\ell)\dotdiv \mult_\Gamma(\ell \bar J)\Big)  &\forall J\subseteq Y,\\
    f(K)&\geq \sum_{u\in Y} \Big (f(u)\dotdiv \mult_\Gamma(u \bar K)\Big)  &\forall K\subseteq [k],\\
     f(J) &\leq f(K) + \mult_\Gamma(J\bar K) &\forall K\subseteq [k], J\subseteq Y,\\  f(K) &\leq f(J) + \mult_\Gamma(K\bar J) &\forall K\subseteq [k], J\subseteq Y.
\end{align*}
But these conditions are equivalent to  those of Theorem \ref{rhohallthm}. 
\end{proof}

\section{{\normalfont Ryser's Theorem for }\protect\boldmath$\rho${\normalfont-latin Rectangles} } \label{rhorythmprfsec}
Seventy years ago, Ryser \cite{MR42361} showed that any $r\times s$  latin rectangle can be extended to an $n\times n$ latin square if and only if  $e_\ell\geq r+s-n$ for $\ell\in [n]$. We provide a generalization of this result to $\brho$-latin rectangles by proving the following Theorem \ref{rhoryserthmfullversion}. For a simpler generalization of Ryser's Theorem, see Corollary \ref{cor2}.

\noindent \textit{Proof of Theorem \ref{rhoryserthmfullversion}.}
Recall that the set of symbols used is $[k]$.  The existence of a fitting sequence $\{(a_\ell,b_\ell)\}_{\ell=1}^k$  implies ~\eqref{mainryserineq}. To see this, observe that by Definition \ref{fitdef},
$$
0\leq a_\ell+b_\ell\leq 2n- r-s+ e_\ell-\rho_\ell \quad \mbox {for } \ell\in P_{\min\{r,s\}}. 
$$
Moreover, using the definition of $P_{\min\{r,s\}}$, we have
$$\rho_\ell-e_\ell \leq n- \min\{r,s\}\leq 2n-r-s \quad \mbox{for } \ell\in \bar P_{\min\{r,s\}}.$$ 
The necessity of all  conditions will be evident at the end of the proof. For now, let us assume that ~\eqref{mainryserineq} --- whose necessity was established in the introduction --- holds. 
Let $F[\tilde X,\tilde Y]$ be the complete bigraph $K_{n,n}$ where $\tilde X:=\{x_1,\dots,x_n\}, \tilde Y:=\{y_1,\dots, y_n\}$, and let $X=\{x_1,\dots, x_r\},Y=\{y_1,\dots,y_s\}$.  The edge $x_i y_j$ of $F$ is colored $\ell$ if $L_{ij}=\ell$ for $i\in[r],j\in[s]$. We have $\dg_{F(\ell)}(u)\leq 1$ for $u\in X\cup Y, \ell\in [k]$,  $e_\ell=|E(F(\ell))|\leq \rho_\ell$ for $\ell\in [k]$, and $\sum_{\ell\in [k]}e_\ell=rs$ (See Figure \ref{fig1rhoryser} (b)).  
Let  $G$ be the bigraph obtained from $F$ by amalgamating $x_{r+1},\dots, x_n$ into a single vertex $\alpha$, and amalgamating $x_{s+1},\dots, y_n$ into a single vertex $\beta$, so $\mult_G(\alpha y_j)=n-r$ for $j\in [s]$, $\mult_G(\beta x_i)=n-s$ for $ i\in [r]$, and $\mult_G(\alpha \beta)=(n-r)(n-s)$ (See Figure \ref{fig1rhoryser} (c)).  Let  $\Gamma [X\cup Y, [k]]$ be the simple bigraph whose edge set is 
$$\{u\ell\ \big |\ u\in X\cup Y, \ell\in [k],  \dg_{F(\ell)}(u)=0\}.$$ 
For  $i\in [r]$, since $\sum_{\ell\in [k]}\dg_{F(\ell)}(x_i)=s$, we have $\dg_\Gamma(x_i)=k-s$. Similarly, for  $j\in [s]$,  $\sum_{\ell\in [k]}
\dg_{F(\ell)}(y_j)=r$, and so we have $\dg_\Gamma(y_j)=k-r$. For $\ell\in [k]$, there are $e_\ell$ edges in $F(\ell)$, and so  $\mult_{\Gamma}(\ell X)=r-e_\ell$ and $\mult_{\Gamma}(\ell Y)=s-e_\ell$ (Recall that by the latin property of $L$, $e_\ell \leq \min \{r,s\}$). Therefore, $\dg_\Gamma(\ell)=r+s-2e_\ell$ for $\ell\in [k]$. In short, $\Gamma$ meets the following properties (See Figure \ref{fig1rhoryser} (d)).
\begin{align*}%\label{nothingspec}
\begin{cases}
\dg_\Gamma(x_i)=k-s & \mbox{if  }  i\in [r], \\
\dg_\Gamma(y_j)=k-r & \mbox{if  }  j\in [s], \\
\mult_{\Gamma}(\ell X)=r-e_\ell  & \mbox{if  } \ell \in [k],\\
\mult_{\Gamma}(\ell Y)=s-e_\ell  & \mbox{if  } \ell \in [k],\\
\dg_\Gamma(\ell)=r+s-2e_\ell  & \mbox{if  } \ell \in [k].
\end{cases}
\end{align*}
\begin{figure}[p] 
  \begin{adjustbox}{addcode={\begin{minipage}{\width}}{\caption{
      A $\rho$-latin rectangle and three associated auxiliary bigraphs
      }\end{minipage}},rotate=90,center}  
 \includegraphics[scale=1.2]{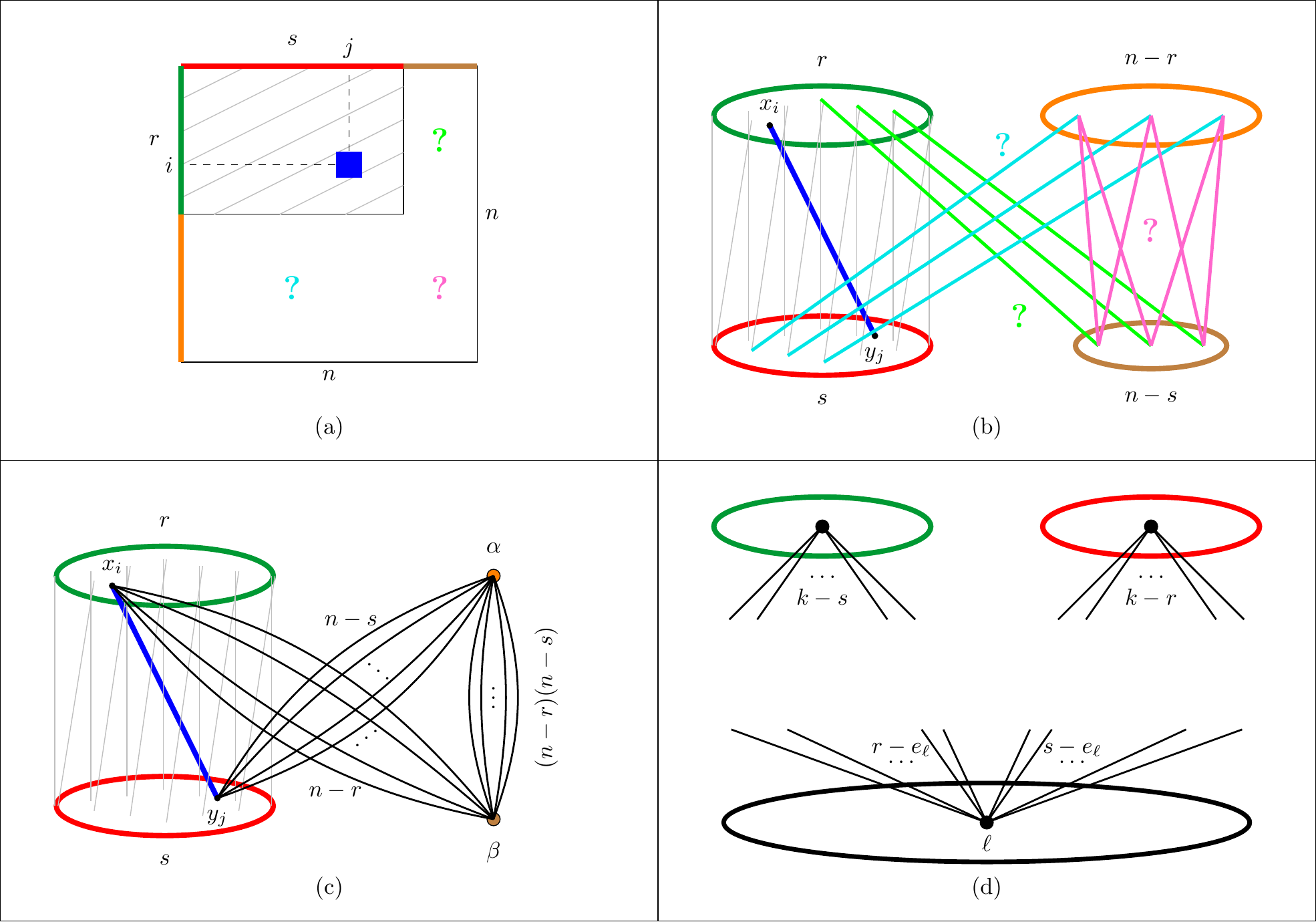} \end{adjustbox}
 \label{fig1rhoryser} 
\end{figure}

Observe that $L$ can be completed if and only if the uncolored edges of  $F$ can be colored so that
\begin{align}\label{colorcon1}
\forall \ell\in [k],\quad \quad \begin{cases}
\dg_{F(\ell)}(u)\leq 1 & \mbox{if  }  u\in \tilde X \cup \tilde Y, \\
|E(F(\ell))|=\rho_\ell.
\end{cases}
\end{align}

We show that the coloring of  $F$ can be completed such that  \eqref{colorcon1} is satisfied if and only if the coloring of $G$ can be completed such that,
\begin{align}\label{colorcon2}
\forall \ell\in [k],\quad \quad  \begin{cases}
\dg_{G(\ell)}(u) \leq 1 & \mbox{if  } u \in X \cup Y, \\
\dg_{G(\ell)}(\alpha) \leq n-r, \\
\dg_{G(\ell)}(\beta) \leq n-s, \\
|E(G(\ell))|= \rho_\ell.
\end{cases}
\end{align}
To see this, first assume that the coloring of  $F$ can be completed so that  \eqref{colorcon1} holds. Identifying all the  vertices in $\tilde X\backslash X$   by $\alpha$ and all the vertices in $\tilde Y\backslash Y$   by $\beta$, we will get the graph $G$ satisfying \eqref{colorcon2}. Conversely, suppose that  we have a  coloring of $G$
such that \eqref{colorcon2} holds. Applying the detachment lemma to $G$, we get a graph $F'$ obtained by splitting $\alpha$ and $\beta$ into $\alpha_1,\dots,\alpha_{n-r}$, and  $\beta_1\dots,\beta_{n-s}$, respectively, such that the following hold.
\begin{enumerate}
\item [(i)] $\dg_{F'(\ell)}(\alpha_i)\approx\dg_{G(\ell)}(\alpha)/(n-r)\leq 1$  for  $i\in [n-r],\ell\in [k]$;
\item [(ii)] $\dg_{F'(\ell)}(\beta_j)\approx \dg_{G(\ell)}(\beta)/(n-s)\leq 1$ for  $j\in [n-s],\ell\in [k]$;
\item [(iii)] $\mult_{F'}(\alpha_i u)= \mult_G(\alpha u)/(n-r)=1$ for  $i\in[n-r],u\in Y$;
\item [(iv)] $\mult_{F'}(\beta_j u)= \mult_G(\beta u)/(n-s)=1$ for  $j\in[n-s],u\in X$;
\item [(v)] $\mult_{F'}(\alpha_i \beta_j)= \mult_G(\alpha\beta)/\big((n-r)(n-s)\big)=1$ for  $i\in[n-r],j\in[n-s]$.
\end{enumerate}
 Since $F'\cong F$ and the coloring of $F'$ satisfies \eqref{colorcon1}, we are done. 

For $\ell\in [k]$, since $\rho_\ell\leq n$, we have that $\rho_\ell-e_\ell-n+r\leq r-e_\ell$ and $\rho_\ell-e_\ell-n+s\leq s-e_\ell$. We show that the coloring of  $G$ can be completed such that  \eqref{colorcon2} is satisfied if and only if there exists  a subgraph $\Theta$ of $\Gamma$  with $n(r+s)-2rs$ edges so that 
\begin{align}\label{colorcon3}
\begin{cases}
\dg_\Theta (x_i)=n-s & \mbox{if  }  i\in [r], \\
\dg_\Theta (y_j)=n-r  & \mbox{if  }  j\in [s], \\
\dg_\Theta (\ell)\leq \rho_\ell-e_\ell  & \mbox{if  } \ell \in [k],\\
\mult_\Theta (\ell X)\geq \rho_\ell-e_\ell-n+r  & \mbox{if  } \ell \in [k],\\
\mult_\Theta (\ell Y)\geq \rho_\ell-e_\ell-n+s  & \mbox{if  } \ell \in [k].
\end{cases}
\end{align}

To prove this, suppose that the coloring $G$ can completed such that  \eqref{colorcon2} is satisfied. Let  $\Theta [X\cup Y, [k]]$ be the bigraph whose edge set is 
$$\{u \ell\ \big |\ u\in X\cup Y, \ell\in [k], \mbox{ and either }\alpha u\in E(G(\ell)) \mbox{ or } \beta u \in E(G(\ell))  \}.$$ 
It is clear that  $\Theta \subseteq \Gamma$.   
Moreover, 
 $\dg_\Theta(x_i)=\mult_G(\beta x_i)=n-s$ for $i\in [r]$, $\dg_\Theta(y_j)=\mult_G(\alpha y_j)=n-r$ for $j\in [s]$, and $\dg_{\Theta}(\ell)\leq |E(G(\ell))|-e_\ell=\rho_\ell-e_\ell$ for $\ell\in [k]$. Notice that 
 $$\rho_\ell=|E(G(\ell))|=e_\ell+\mult_{G(\ell)}(\alpha \beta) +\mult_{G(\ell)} (\beta X) + \mult_{G(\ell)} (\alpha Y)\quad  \mbox {for }\ell\in [k].$$
Thus,  for $\ell\in [k]$ we have 
 \begin{align*}
\mult_\Theta (\ell X) &=\mult_{G(\ell)} (\beta X)   \\
& = \rho_\ell-e_\ell-\mult_{G(\ell)}(\alpha \beta) - \mult_{G(\ell)} (\alpha Y)\\
 &=  \rho_\ell-e_\ell-\dg_{G(\ell)}(\alpha) \\
& \geq   \rho_\ell-e_\ell-n+r,\\
\mult_\Theta (\ell Y)& =\mult_{G(\ell)} (\alpha Y)   \\
 &=  \rho_\ell-e_\ell-\mult_{G(\ell)}(\alpha \beta) - \mult_{G(\ell)} (\beta X)\\
& = \rho_\ell-e_\ell-\dg_{G(\ell)}(\beta) \\
& \geq   \rho_\ell-e_\ell-n+s,
\end{align*}
and so \eqref{colorcon3} holds. Observe  that $\rho_\ell-e_\ell  \geq \dg_\Theta (\ell)= \mult_\Theta (\ell X)+ \mult_\Theta (\ell Y)\geq  \rho_\ell-e_\ell-n+r +  \rho_\ell-e_\ell-n+s$; this is  consistent with   ~\eqref{mainryserineq}.

Conversely, suppose that   $\Theta\subseteq \Gamma$ satisfying \eqref{colorcon3} exists. For each $\ell\in [k]$, if $\ell x_i \in E(\Theta)$ for some $i\in [r]$, we color a $\beta x_i$-edge in $G$ with $\ell$, and  if $\ell y_j \in E(\Theta)$ for some $j\in [s]$, we color an $\alpha y_j$-edge in $G$ with $\ell$. Since $\dg_\Theta(x_i)=n-s$ for $i\in [r]$, and $\dg_\Theta(y_j)=n-r$ for $j\in [s]$, all the edges between $\beta$ and $X$, and all the edges between $\alpha$ and $Y$ can be colored this way. Since $\Theta$ is simple, $d_{G(\ell)}(u)\leq 1$ for $\ell\in [k]$ and $u\in X\cup Y$.  Then, we color the $\alpha \beta$-edges so that
$$\mult_{G(\ell)}(\alpha \beta) = \rho_\ell-e_\ell-\dg_\Theta (\ell)\quad  \mbox {for }\ell\in [k].$$
This can be done, because $ \rho_\ell-e_\ell-\dg_\Theta (\ell)\geq 0$ for $\ell\in [k]$, and 
\begin{align*}
   \sum_{\ell\in [k]}\mult_{G(\ell)}(\alpha \beta)&=n^2-rs-n(r+s)+2rs\\
   &=(n-r)(n-s)\\&=\mult_G(\alpha \beta).
\end{align*}
For $\ell\in [k]$, we have
  \begin{align*}
\dg_{G(\ell)} (\alpha) &=\mult_{G(\ell)} (\alpha Y)+ \mult_{G(\ell)} (\alpha \beta)  \\
 &=  \mult_\Theta (\ell Y)+\rho_\ell-e_\ell-\dg_\Theta (\ell)\\
 &=  \rho_\ell-e_\ell-\mult_\Theta (\ell X)\\
& \leq  n-r,\\
\dg_{G(\ell)} (\beta) &=\mult_{G(\ell)} (\beta X)+ \mult_{G(\ell)} (\alpha \beta)  \\
& =  \mult_\Theta (\ell X)+\rho_\ell-e_\ell-\dg_\Theta (\ell)\\
& =  \rho_\ell-e_\ell-\mult_\Theta (\ell Y)\\
& \leq  n-s,
\end{align*}
and so \eqref{colorcon2} is satisfied.

Let $\Gamma_1, \Gamma_2$ be the subgraphs of $\Gamma$ induced by $X\cup [k]$, and  $Y\cup [k]$, respectively. The existence of $\Theta\subseteq \Gamma$ satisfying \eqref{colorcon3} is equivalent to the existence of $\Theta_1\subseteq \Gamma_1, \Theta_2\subseteq \Gamma_2$ satisfying the following condition.
\begin{align*}%\label{colorconfinala}
\begin{cases}
\dg_{\Theta_1} (x_i)=n-s & \mbox{if  }  i\in [r], \\
\dg_{\Theta_1} (\ell)\geq \rho_\ell-e_\ell-n+r  & \mbox{if  } \ell \in [k],\\
\dg_{\Theta_2} (y_j)=n-r  & \mbox{if  }  j\in [s], \\
\dg_{\Theta_2} (\ell)\geq \rho_\ell-e_\ell-n+s  & \mbox{if  } \ell \in [k],\\
   \dg_{\Theta_1} (\ell)+\dg_{\Theta_2} (\ell)\leq \rho_\ell-e_\ell   &  \mbox{if  } \ell \in [k].
\end{cases}
\end{align*}
The three inequalities of  this  condition are equivalent to the existence of a sequence $\{(a_\ell,b_\ell)\}_{\ell=1}^k$ with $a_\ell,b_\ell\in \mathbb{N} \cup \{0\}$ for $\ell\in [k]$ such that the following three properties are satisfied (Recall that $P_{\min \{r,s\}}\subseteq P_{\max \{r,s\}}$).

\begin{align*}
\dg_{\Theta_1} (\ell)&=\begin{cases}
a_\ell+\rho_\ell-e_\ell-n+r  & \mbox{if  } \ell \in P_r,\\
a_\ell  & \mbox{if  } \ell \in \bar P_r,
\end{cases}\\
\dg_{\Theta_2} (\ell)&=\begin{cases}
b_\ell+\rho_\ell-e_\ell-n+s  & \mbox{if  } \ell \in P_s,\\
b_\ell  & \mbox{if  } \ell \in \bar P_s,
\end{cases}\\
\rho_\ell-e_\ell&\geq
\begin{cases}
a_\ell+b_\ell+2\rho_\ell-2e_\ell-2n+r+s  & \mbox{if  } \ell \in P_{\min \{r,s\}},\\
a_\ell+b_\ell+\rho_\ell-e_\ell-n+\max \{r,s\}  & \mbox{if  } \ell \in P_{\max \{r,s\}}\backslash P_{\min \{r,s\}},\\
a_\ell+b_\ell  & \mbox{if  } \ell \in \bar P_{\max \{r,s\}},
\end{cases}
\end{align*}
where the last property simplifies to the following.
\begin{align} \label{abseqineqpro}
        a_\ell+b_\ell\leq \begin{cases}
            2n- r-s+ e_\ell-\rho_\ell& \mbox {if  } \ell \in P_{\min \{r,s\}},\\
            n- \max \{r,s\} & \mbox {if  } \ell\in \ell \in P_{\max \{r,s\}}\backslash P_{\min \{r,s\}},\\
            \rho_\ell -e_\ell & \mbox {if  } \ell \in \bar P_{\max \{r,s\}}.
                \end{cases}
\end{align} 
We define $f_1,f_2$ as follows. 
\begin{align*}
    \begin{cases}
        f_1: V(\Gamma_1)\rightarrow \mathbb{N}\cup \{0\},\\
        f_1(x_i)=n-s & \mbox{if  } i\in [r],\\
        f_1(\ell)=a_\ell+\rho_\ell-e_\ell-n+r & \mbox{if  }\ell\in P_r, \\      
        f_1(\ell)=a_\ell & \mbox{if  }\ell\in \bar P_r,       
    \end{cases}
\end{align*}
\begin{align*}
    \begin{cases}
        f_2:  V(\Gamma_2)\rightarrow \mathbb{N}\cup \{0\},\\
        f_2(y_j)=n-r & \mbox{for  } i\in [s],\\
        f_2(\ell)=b_\ell+\rho_\ell-e_\ell-n+s & \mbox{for  }\ell\in P_s,  \\ 
        f_2(\ell)=b_\ell & \mbox{for  }\ell\in \bar P_s.  
    \end{cases}
\end{align*}

By the preceding paragraph,  $\Theta\subseteq \Gamma$ satisfying \eqref{colorcon3} exists if and only if there exists a sequence $\{(a_\ell,b_\ell)\}_{\ell=1}^k$  satisfying \eqref{abseqineqpro} such that  $\Gamma_1$ and $\Gamma_2$ have an $f_1$-factor and $f_2$-factor, respectively. 
Observe that $f_1(X)=f_1([k])$ if and only if
\begin{align*}
    r(n-s)=\sum_{\ell\in P_r} (a_\ell+\rho_\ell-e_\ell-n+r) + \sum_{\ell\in \bar P_r} a_\ell,
\end{align*}
or equivalently,  
\begin{align} \label{abseqsum1}
     \sum_{\ell\in [k]} a_\ell=r(n-s)+|P_r|(n-r)-\sum_{\ell\in P_r} (\rho_\ell-e_\ell).
\end{align}
Similarly, $f_2(Y)=f_2([k])$ if and only if
\begin{align*}
    s(n-r)=\sum_{\ell\in P_s} (b_\ell+\rho_\ell-e_\ell-n+s) + \sum_{\ell\in \bar P_s} b_\ell,
\end{align*}
or equivalently,   
\begin{align}\label{abseqsum2}
     \sum_{\ell\in [k]} b_\ell=s(n-r)+|P_s|(n-s)-\sum_{\ell\in P_s} (\rho_\ell-e_\ell).
\end{align}

By Ore's Theorem, $\Gamma_1$ has an $f_1$-factor  if and only if \eqref{abseqsum1} together with any  of the following conditions hold.
\begin{align*}
    f_1(I)&\leq \sum_{\ell\in [k]} \min\Big\{f_1(\ell), \mult_{\Gamma_1}(\ell I)\Big\} &\forall I\subseteq X,\\
    f_1(K)&\leq \sum_{u\in X} \min\Big\{f_1(u), \mult_{\Gamma_1}(uK)\Big\} &\forall K\subseteq [k],\\
    f_1(I)&\geq \sum_{\ell\in [k]} \Big (f_1(\ell)\dotdiv \mult_{\Gamma_1}(\ell \bar I)\Big)  &\forall I\subseteq X,\\
    f_1(K)&\geq \sum_{u\in X} \Big (f_1(u)\dotdiv \mult_{\Gamma_1}(u \bar K)\Big)  &\forall K\subseteq [k],\\
     f_1(I) &\leq f_1(K) + \mult_{\Gamma_1}(I\bar K) &\forall K\subseteq [k], I\subseteq X,\\
    f_1(K) &\leq f_1(I) + \mult_{\Gamma_1}(K\bar I) &\forall K\subseteq [k], I\subseteq X,
\end{align*} 
Similarly,  $\Gamma_2$ has an $f_2$-factor  if and only if \eqref{abseqsum2}  and any  of the following conditions hold.
\begin{align*}
    f_2(J)&\leq \sum_{\ell\in [k]} \min\Big\{f_2(\ell), \mult_{\Gamma_2}(\ell J)\Big\} &\forall J\subseteq Y,\\
    f_2(K)&\leq \sum_{u\in Y} \min\Big\{f_2(u), \mult_{\Gamma_2}(uK)\Big\} &\forall K\subseteq [k],\\
    f_2(J)&\geq \sum_{\ell\in [k]} \Big (f_2(\ell)\dotdiv \mult_{\Gamma_2}(\ell \bar J)\Big)  &\forall J\subseteq Y,\\
    f_2(K)&\geq \sum_{u\in Y} \Big (f_2(u)\dotdiv \mult_{\Gamma_2}(u \bar K)\Big)  &\forall K\subseteq [k],\\
     f_2(J) &\leq f_2(K) + \mult_{\Gamma_2}(J\bar K) &\forall K\subseteq [k], J\subseteq Y,\\
    f_2(K) &\leq f_2(J) + \mult_{\Gamma_2}(K\bar J) &\forall K\subseteq [k], J\subseteq Y.
\end{align*} 
These conditions are equivalent to those of Theorem \ref{rhoryserthmfullversion}.
\qed
\begin{remark} \textup{
\begin{itemize}
    \item [(a)] Let us verify that Theorem \ref{rhoryserthmfullversion} implies Theorem \ref{rhohallthm}.  In Theorem \ref{rhoryserthmfullversion}, let $s=n$. Recall that the existence of the fitting sequence $\{(a_\ell,b_\ell)\}_{\ell=1}^k$ implies  that $\rho_\ell-e_\ell \leq n-r$ for $\ell\in [k]$. Therefore,  $P_r=\emptyset, P_n=\{\ell\in [k] \ |\ \rho_\ell>e_\ell\}$. We have
    \begin{align*}
    \begin{cases}
       \displaystyle\sum_{\ell\in [k]} a_\ell=0,\\ \\
        \displaystyle\sum_{\ell\in [k]} b_\ell=n(n-r)-\sum_{\ell\in P_n}(\rho_\ell-e_\ell)=n(n-r)-\sum_{\ell\in [k]}(\rho_\ell-e_\ell)=0,\\ 
         a_\ell+b_\ell\leq 
        \begin{cases}
            n- r+ e_\ell-\rho_\ell& \mbox {if  } \ell \in \emptyset,\\
            0 & \mbox {if  } \ell \in P_{n},\\
            \rho_\ell -e_\ell=0 & \mbox {if  } \ell \in \bar P_{n}.
        \end{cases}
    \end{cases}
    \end{align*}     
    Hence, there is a unique fitting sequence, namely, $a_\ell=b_\ell=0$ for $\ell\in [k]$.
Consequently, the first six conditions of Theorem \ref{rhoryserthmfullversion} will be trivial, and the remaining six conditions will be identical to those of Theorem \ref{rhohallthm}. 
\item [(b)] Now, we show that Ryser's Theorem is  immediate from Theorem \ref{rhoryserthmfullversion}. In Theorem \ref{rhoryserthmfullversion},  if we let $k=n$ (so $L$ is a latin rectangle), then  the existence of $\Theta\subseteq \Gamma$ satisfying \eqref{colorcon3} is trivial. In fact, in this case, $\Theta=\Gamma$. 
%\item [(c)] 
\end{itemize}
}\end{remark}
\begin{remark} \textup{
We can restate Theorem \ref{rhoryserthmfullversion} without introducing the sets $P_r$ and $P_s$. To do so, let us  rewrite $f_1(\ell)=a_\ell+(\rho_\ell-e_\ell+r\dotdiv n)$, and  $f_2(\ell)=b_\ell+(\rho_\ell-e_\ell+s\dotdiv n)$. The main condition in the definition of fitting sequence can be written in the following manner.
\begin{align*}
\begin{cases}
       \displaystyle\sum_{\ell\in [k]} a_\ell=r(n-s)-\sum_{\ell\in [k]}(\rho_\ell-e_\ell+r\dotdiv n),\\[20pt] 
        \displaystyle\sum_{\ell\in [k]} b_\ell=s(n-r)-\sum_{\ell\in [k]}(\rho_\ell-e_\ell+s\dotdiv n),\\[20pt]  
         a_\ell+b_\ell\leq \rho_\ell-e_\ell-(\rho_\ell-e_\ell+r\dotdiv n)-(\rho_\ell-e_\ell+s\dotdiv n).
    \end{cases}
\end{align*}
The first six conditions of Theorem \ref{rhoryserthmfullversion} can be restated as follows.
}\end{remark}
\begin{align*}
    |I|(n-s) &\leq \sum_{\ell\in [k]} \min \Big\{a_\ell+(\rho_\ell-e_\ell+r\dotdiv n), \mu_I(\ell)\Big\}  &\forall I\subseteq  [r],\\
\sum_{i\in [r]} \min \Big\{n-s, \mu_K(i)\Big\}  &\geq     \sum_{\ell\in K} \Big(a_\ell+(\rho_\ell-e_\ell+r\dotdiv n)\Big)&\forall K\subseteq  [k],\\
  |I|(n-s) &\geq \sum_{\ell\in [k]} \Big( a_\ell+(\rho_\ell-e_\ell+r \dotdiv n)\dotdiv  \mu_{\bar I}(\ell)\Big) &\forall I\subseteq [r],\\ 
\sum_{i\in [r]} \Big (n-s\dotdiv \mu_{\bar K}(i)\Big)  &\leq  \sum_{\ell\in K}\Big(a_\ell+(\rho_\ell-e_\ell+r \dotdiv n)\Big)     &\forall K\subseteq [k],\\
    |I|(n-s) &\leq \sum_{\ell\in K}\Big(a_\ell+(\rho_\ell-e_\ell+r \dotdiv n)\Big)  + \mu_I(\bar K)  &\forall I\subseteq [r], K\subseteq [k], \\
  |I|(n-s)  &\geq   \sum_{\ell\in K}\Big(a_\ell+(\rho_\ell-e_\ell+r \dotdiv n)\Big) - \mu_K(\bar I)   &\forall I\subseteq  [r], K\subseteq [k].
\end{align*}
Substituting $r$ by $s$, $s$ by $r$, $I$ by $J$, and $a_\ell$ by $b_\ell$ in these conditions, one can obtain the remaining six conditions. 
\section{Corollaries} 
Throughout this section, we will use the same notation as in 
Section \ref{rhorythmprfsec}.  Recall that in order to embed an $r\times s$  $\brho$-latin rectangle into an $n\times n$ $\brho$-latin square, it is necessary that
$$
e_\ell\geq r+s+\rho_\ell-2n \quad \forall \ell\in [k].
$$
We show that imposing slightly stronger assumptions will lead to much simpler conditions than those of Theorem \ref{rhoryserthmfullversion}. Most importantly, the sequence $\{(a_\ell,b_\ell)\}_{\ell=1}^k$ will not be needed if we allow each $e_\ell$ to be slightly bigger  than what  is necessary.

We will  make use of the following $(g,f)$-factor theorem. Here, $N_G(A)$ in the neighborhood of $A$ in $G$. 
\begin{theorem} \cite[Theorem 5]{MR3564794}, \cite[Theorem 1]{MR1081839} \label{gffacthmcomb}
The bipgraph $G[X,Y]$ has a  $(g,f)$-factor   if and only if either one of the following two conditions hold.
\begin{align*}
    g(A)&\leq \sum_{u\in N_G(A)} \min \Big\{ f(u), \mult_G(uA)\Big\}
 &\forall A\subseteq X, A\subseteq Y,\\
f(A)&\geq  \sum_{u\notin A} \Big( g(u) \dotdiv \dg_{G-A}(u) \Big)  &\forall A\subseteq X\cup Y.
\end{align*}
\end{theorem} 
Slight modification to the proof of \cite[Theorem 1]{MR1081839} leads to the following  simpler criteria for the case when  $g(y)=0$ for $y\in Y$: The  bigraph $G[X,Y]$ has a  $(g,f)$-factor  with $g(y)=0$ for $y\in Y$  if and only if
\begin{equation} \label{refinedgfthmmain}
    f(B)\geq  \sum_{x\in A} \Big( g(x) \dotdiv \dg_{G-B}(x) \Big) \quad \forall A\subseteq X, B\subseteq Y.
\end{equation}

In our first application of Theorem \ref{rhoryserthmfullversion}, we assume that $e_\ell\geq \rho_\ell-n+\max\{r,s\}$ for  $\ell\in [k]$. This, in particular implies that
$$
rs=\sum_{\ell\in [k]}e_\ell\geq \sum_{\ell\in [k]} \big( \rho_\ell-n+\max\{r,s\}\big)=n^2-kn + k \max\{r,s\},
$$
which leads to the following lower bound for the number of symbols.
$$k\geq \frac{n^2-rs}{n-\max\{r,s\}}\geq n+\max\{r,s\}.$$

\begin{corollary} \label{corkbig}
An $r\times s$  $\brho$-latin rectangle with $e_\ell\geq \rho_\ell-n+\max\{r,s\}$ for  $\ell\in [k]$ can  be completed to an $n\times n$ $\brho$-latin square if and only if any of the following conditions hold.
\begin{align*}
 |I|(n-s)+ |J|(n-r)&\leq \sum_{\ell\in [k]} \min \Big\{ \rho_\ell-e_\ell, \mu_I(\ell) + \mu_J(\ell)\Big\}
 &\forall I\subseteq [r], J\subseteq [s],\\
\sum_{\ell\in K} (\rho_\ell-e_\ell)&\geq \sum_{i\in  I} \Big( n-s \dotdiv \mu_{\bar K}(i) \Big)+
  \sum_{j\in  J} \big( n-r \dotdiv \mu_{\bar K}(j) \Big) &\forall I\subseteq [r], J\subseteq [s], K\subseteq [k].
\end{align*}
\end{corollary}
\begin{proof}
Suppose that $e_\ell\geq \rho_\ell-n+\max\{r,s\}$ for $\ell\in [k]$. Then \eqref{mainryserineq} holds, and neither $\rho_\ell-e_\ell-n+s$ nor $\rho_\ell-e_\ell-n+r$ is positive for $\ell\in [k]$. Thus,    \eqref{colorcon3} will be simplified to the following.
\begin{align*}
\begin{cases}
\dg_\Theta (x_i)=n-s & \mbox{if  }  i\in [r], \\
\dg_\Theta (y_j)=n-r  & \mbox{if  }  j\in [s], \\
\dg_\Theta (\ell)\leq \rho_\ell-e_\ell  & \mbox{if  } \ell \in [k].
\end{cases}
\end{align*}
 Let 
\begin{align*}
    \begin{cases}
        g,f: V(\Gamma)\rightarrow \mathbb{N}\cup \{0\},\\
        g(x_i)=f(x_i)=n-s & \mbox {for } i\in [r],\\
        g(y_j)=f(y_j)=n-r & \mbox {for } j\in [s],\\
        g(\ell)=0  & \mbox {for } \ell\in [k],\\
        f(\ell)=\rho_\ell-e_\ell & \mbox {for } \ell\in [k].
    \end{cases}
\end{align*} 
By Theorem \ref{gffacthmcomb}, $\Gamma  [X\cup Y, [k]]$ has a $(g,f)$-factor (and so $\Theta\subseteq \Gamma$ satisfying \eqref{colorcon3} exists, and consequently, $L$ can be completed) if and only if the following  conditions hold.
\begin{align*}
 g(I)+ g(J)&\leq \sum_{\ell\in [k]} \min \Big\{ f(\ell), \mult_\Gamma(\ell I)+ \mult_\Gamma(\ell J)\Big\}
 &\forall I\subseteq X, J\subseteq Y,\\
 g(K)&\leq \sum_{u\in X\cup Y} \min \Big\{ f(u), \mult_\Gamma(u K)\Big\}
 &\forall K\subseteq [k].
\end{align*}
The first condition is equivalent to the first condition of this corollary. The second condition is trivial for $g(K)=0$ for $K\subseteq [k]$. By \eqref{refinedgfthmmain}  $\Gamma$ has a  $(g,f)$-factor if and only if
\begin{equation*} %\label{refinedgfthm}
    f(K)\geq  \sum_{u\in U} \Big( g(u) \dotdiv \dg_{\Gamma-K}(u) \Big) \quad \forall U\subseteq (X\cup Y), K\subseteq [k].
\end{equation*}
This  is equivalent to the second condition of this corollary.
\end{proof}

In our next application, we assume that $e_\ell\geq r+s-\rho _\ell$ for  $\ell\in [k]$, which implies the following.
$$
rs=\sum_{\ell\in [k]}e_\ell\geq \sum_{\ell\in [k]} \big( r+s-\rho _\ell\big)=k(r+s)-n^2,
$$
This leads to the following upper bound for the number of symbols.
$$k\leq \frac{n^2+rs}{r+s}.$$

In the next corollary, we use the fact that $(x\dotdiv y)\dotdiv z=x-y \dotdiv z$.
\begin{corollary} \label{cor2}
An $r\times s$  $\brho$-latin rectangle with $e_\ell\geq r+s-\rho _\ell$ for  $\ell\in [k]$ can  be completed to an $n\times n$ $\brho$-latin square if and only if any of  the following  two conditions,
\begin{align*}
    \sum_{\ell\in K} (\rho_\ell-e_\ell+r\dotdiv n) &\leq \sum_{i\in [r]} \min \Big\{ n-s, \mu_K(i)\Big\}  & \forall K\subseteq [k],\\
|I|(n-s)   &\geq \sum_{\ell\in [k]} \big( \rho_\ell-e_\ell+r- n \dotdiv \mu_{\bar I}(\ell) \big)  & \forall I\subseteq [r],
\end{align*}
together with any of the following two  hold.
\begin{align*}
    \sum_{\ell\in K} (\rho_\ell-e_\ell+s\dotdiv n) &\leq \sum_{j\in [s]} \min \Big\{ n-r, \mu_K(j)\Big\}  & \forall K\subseteq [k],\\
|J|(n-r)   &\geq \sum_{\ell\in  [k]} \big( \rho_\ell-e_\ell+s- n \dotdiv \mu_{\bar J}(\ell) \big) &  \forall J\subseteq [s].
\end{align*}
\end{corollary}
\begin{proof} 
Suppose that $e_\ell\geq r+s-\rho _\ell$ for $\ell\in [k]$. Then \eqref{mainryserineq} holds,  $\rho_\ell-e_\ell\geq r+s-2e_\ell$ for $\ell\in [k]$, and so  \eqref{colorcon3} will be simplified to the following.
\begin{align*}
\begin{cases}
\dg_\Theta (x_i)=n-s & \mbox{if  }  i\in [r], \\
\dg_\Theta (y_j)=n-r  & \mbox{if  }  j\in [s], \\
\mult_\Theta (\ell X)\geq \rho_\ell-e_\ell-n+r  & \mbox{if  } \ell \in [k],\\
\mult_\Theta (\ell Y)\geq \rho_\ell-e_\ell-n+s  & \mbox{if  } \ell \in [k].
\end{cases}
\end{align*}
Let $\Gamma_1, \Gamma_2$ be the subgraphs of $\Gamma$ induced by $X\cup [k]$, and  $Y\cup [k]$, respectively. Let  
\begin{align*}
    \begin{cases}
        g_1, f_1,: V(\Gamma_1)\rightarrow \mathbb{N}\cup \{0\},\\
        f_1(x_i)=g_1(x_i)=n-s &\mbox {for } i\in [r],\\
        g_1(\ell)=\rho_\ell-e_\ell+r\dotdiv n &\mbox {for } \ell\in [k],\\
        f_1(\ell)=r-e_\ell &\mbox {for } \ell\in [k],
    \end{cases}
\end{align*}
\begin{align*}
    \begin{cases}
        g_2, f_2,: V(\Gamma_2)\rightarrow \mathbb{N}\cup \{0\},\\
        f_2(y_j)=g_2(y_j)=n-r &\mbox {for } j\in [s],\\
        g_2(\ell)=\rho_\ell-e_\ell+s\dotdiv n &\mbox {for } \ell\in [k],\\
        f_2(\ell)=s-e_\ell &\mbox {for } \ell\in [k].
    \end{cases}
\end{align*}
It is clear that  $\Theta\subseteq \Gamma$ satisfying \eqref{colorcon3} exists if and only if $\Gamma_1$ and $\Gamma_2$ have an $f_1$-factor and $f_2$-factor, respectively. By Theorem  \ref{gffacthmcomb}, $\Gamma_1$ has a $(g_1,f_1)$-factor if and only if
\begin{align*}
    g_1(K)&\leq \sum_{u\in X} \min \Big\{ f_1(u), \mult_{\Gamma_1}(u K)\Big\}
 &\forall K\subseteq [k],\\
    g_1(I)&\leq \sum_{\ell\in [k]} \min \Big\{ f_1(\ell), \mult_{\Gamma_1}(\ell I)\Big\}
 &\forall I\subseteq X,
\end{align*}
and $\Gamma_2$ has a $(g_2,f_2)$-factor if and only if
\begin{align*}
    g_2(K)&\leq \sum_{u\in Y} \min \Big\{ f_2(u), \mult_{\Gamma_2}(u K)\Big\}
 &\forall K\subseteq [k],\\
    g_2(J)&\leq \sum_{\ell\in [k]} \min \Big\{ f_2(\ell), \mult_{\Gamma_2}(\ell J)\Big\}
\quad &\forall J\subseteq Y.
\end{align*}
Equivalently, $\Gamma_1$ has a $(g_1,f_1)$-factor if and only if
\begin{align*}
    \sum_{\ell\in K} (\rho_\ell-e_\ell+r\dotdiv n)&\leq \sum_{i\in [r]} \min \Big\{ n-s, \mu_{K}(i)\Big\}
 &\forall K\subseteq [k],\\
    |I|(n-s)&\leq \sum_{\ell\in [k]} \min \Big\{ r-e_\ell, \mu_I(\ell)\Big\}
 &\forall I\subseteq [r].
\end{align*}
But the second condition is trivial, for $\mu_I(\ell)\leq r-e_\ell$, and so   $\sum_{\ell\in [k]}\mu_I(\ell)=\mu_I([k])=|I|(n-s)$. Similarly, $\Gamma_2$ has a $(g_2,f_2)$-factor if and only if
\begin{align*}
    \sum_{\ell\in K} (\rho_\ell-e_\ell+s\dotdiv n)&\leq \sum_{j\in [s]} \min \Big\{ n-r, \mu_{K}(j)\Big\}
 &\forall K\subseteq [k],\\
    |J|(n-r)&\leq \sum_{\ell\in [k]} \min \Big\{ s-e_\ell, \mu_J(\ell)\Big\}
 &\forall J\subseteq [s].
\end{align*}
 Again, since $\mu_J(\ell)\leq s-e_\ell$, we have  $\sum_{\ell\in [k]}\mu_J(\ell)=\mu_J([k])=|J|(n-r)$, and so the second condition is trivial.

Now we modify $f_1$ and $f_2$ defined above so that for $\ell\in [k]$, $f_1(\ell)=f_2(\ell)=z$ where $z$ is a sufficiently large number. By Theorem \ref{gffacthmcomb}, $\Gamma_1$ has a $(g_1,f_1)$-factor if and only if 
\begin{align} \label{gf1faccond1}
f_1(I)+f_1(K)&\geq    \sum_{\ell\in  \bar K} \Big( g_1(\ell) \dotdiv \dg_{\Gamma_1-(I\cup K)}(\ell) \Big) \nonumber\\
&\quad+ \sum_{i\in  \bar I} \Big( g_1(i) \dotdiv \dg_{\Gamma_1-(I\cup K)}(i) \Big) \quad \forall I\subseteq X, K\subseteq [k].
\end{align}
For $K\neq \emptyset$, \eqref{gf1faccond1} is trivial, and for $K=\emptyset$, it simplifies to the following.
\begin{align*} 
|I|(n-s)&\geq    \sum_{\ell\in   [k]} \Big( (\rho_\ell-e_\ell+r\dotdiv n) \dotdiv \mu_{\bar I}(\ell) \Big)+ \sum_{i\in  \bar I} \Big( n-s  \dotdiv \mu_{[k]}(i) \Big)\\
&=\sum_{\ell\in   [k]} \Big( (\rho_\ell-e_\ell+r\dotdiv n) \dotdiv \mu_{\bar I}(\ell) \Big) \quad \forall I\subseteq [r].
\end{align*}
Similarly,  $\Gamma_2$ has a $(g_2,f_2)$-factor if and only if
\begin{align*} 
|J|(n-r)\geq  \sum_{\ell\in   [k]} \Big( (\rho_\ell-e_\ell+s\dotdiv n) \dotdiv \mu_{\bar J}(\ell) \Big) \quad \forall J\subseteq [s].
\end{align*}
\end{proof}

\begin{corollary}
An $r\times s$  $\brho$-latin rectangle satisfying the following condition can always be completed to an $n\times n$ $\brho$-latin square.  
$$r+s-e_\ell \leq \rho_\ell \leq e_\ell+n-\max\{r,s\}\quad \forall \ell\in [k].$$
\end{corollary}
\begin{proof}
Condition \eqref{mainryserineq} holds, and  we have $\rho_\ell-e_\ell-n+r\leq 0$, $\rho_\ell-e_\ell-n+s\leq 0$, and $\rho_\ell-e_\ell\geq r+s-2e_\ell$ for $\ell\in [k]$. Hence,  ~\eqref{colorcon3}  will be simplified to the following.
\begin{align*}
\begin{cases}
\dg_\Theta (x_i)=n-s & \mbox{if  }  i\in [r], \\
\dg_\Theta (y_j)=n-r  & \mbox{if  }  j\in [s].
\end{cases}
\end{align*}
Consequently,  $\Theta\subseteq \Gamma$ always  exists. 
\end{proof}

\section*{Acknowledgement}
 We wish to thank J. L.  Goldwasser for a fruitful discussion that led to this publication, and the anonymous referees for their constructive criticism.

\end{document}